\documentclass[11pt]{article}
\usepackage{amsmath,amssymb,amsthm}
\usepackage{epsfig}
\usepackage{color}
\usepackage{enumerate}
\usepackage{algorithm}
\usepackage{algorithmic}
\usepackage{hyperref}
\usepackage{enumerate}
\usepackage{graphicx}
\usepackage{tikz}
\usepackage{tkz-graph}
\usepackage{tkz-berge}
\usetikzlibrary{arrows.meta,shapes.arrows}
\hypersetup{colorlinks=true, linkcolor=blue, citecolor=blue,
urlcolor=blue}

\hypersetup{colorlinks=true}

\hypersetup{colorlinks=true, linkcolor=blue, citecolor=blue,urlcolor=blue}

\title{On three outer-independent domination related parameters in graphs\footnote{This work was done while the third author (B. Samadi) was visiting the University of Cadiz, supported by the University of Mazandaran and the Ministry of Science, Research and Technology of Iran, under the program Ph.D. Research Opportunity.}}

\author{Doost Ali Mojdeh$^1$, Iztok Peterin$^2$, Babak Samadi$^3$ and Ismael G. Yero$^4$\\[0.5cm]
Department of Mathematics, University of Mazandaran,\\ Babolsar, Iran$^{1,3}$\\
{\it damojdeh@umz.ac.ir$^1$,} {\it samadibabak62@gmail.com$^3$}\\[0.2cm]
Faculty of Electrical Engineering and Computer Science, University of Maribor,\\
Maribor, Slovenia\\
{\it iztok.peterin@um.si$^2$}\\[0.2cm]
Departamento de Matem\'{a}ticas, Universidad de C\'{a}diz,\\ Algeciras, Spain$^4$\\
{\it ismael.gonzalez@uca.es$^4$}}
\date{}

\newtheorem{theorem}{Theorem}[section]
\newtheorem{corollary}[theorem]{Corollary}
\newtheorem{lemma}[theorem]{Lemma}

\newtheorem{proposition}[theorem]{Proposition}

\newtheorem{p}{Problem}

\theoremstyle{definition}

\theoremstyle{remark}

\begin{document}

\maketitle

\begin{abstract}
Let $G$ be a graph and let $S\subseteq V(G)$. The set $S$ is a double outer-independent dominating set of $G$ if $|N[v]\cap D|\geq2$, for all $v\in V(G)$, and $V(G)\setminus S$ is independent. Similarly, $S$ is a $2$-outer-independent dominating set, if every vertex from $V(G)\setminus S$ has at least two neighbors in $S$ and $V(G)\setminus S$ is independent. Finally, $S$ is a total outer-independent dominating set if every vertex from $V(G)$ has a neighbor in $S$ and the complement of $S$ is an independent set. The double, total or $2$-outer-independent domination number of $G$ is the smallest possible cardinality of any double, total or $2$-outer-independent dominating set of $G$, respectively. In this paper, the $2$-outer-independent, the total outer-independent and the double outer-independent domination numbers of graphs are investigated. We prove some Nordhaus-Gaddum type inequalities, derive their computational complexity and present several bounds for them.
\end{abstract}
\vspace{1mm}


\section{Introduction and preliminaries}

Problems concerning vertex domination and vertex independence in graphs are nowadays very common research topics on the graph theory community. There is a high number of open problems in this area of graph theory, and a rich literature on this issue can be found. Such problems deal with a wide range of situations, which cover theoretical points of view as well as several practical applications in real life situations. One of the most remarkable aspects of domination and/or independence in graphs involves the close relationship between them, throughout the study of different parameters which relate them in several distinct ways. Probably, the most common combination of domination and vertex independence of graphs concerns the independent dominating sets, \emph{i.e.}, dominating sets which are also independent. In this work, we center our attention into studying three different ways of combining domination with independence in graphs. Such combinations are related with finding the sets of smallest cardinality of a domination related parameter whose complement is an independent set.

Throughout this paper, we consider $G$ as a finite simple graph with vertex set $V(G)$ and edge set $E(G)$. We follow \cite{w} as a reference for terminology and notation which is not explicitly defined here. The {\em open neighborhood} of a vertex $v$ consists of all neighbors of $v$ and is denoted by $N_G(v)$. The {\em closed neighborhood} of $v$ is $N_G[v]=N_{G}(v)\cup \{v\}$. The {\em minimum} and {\em maximum degree} of $G$ are denoted by $\delta(G)$ and $\Delta(G)$, respectively. A vertex of degree one is called a {\em leaf} and a neighbor of a leaf is its {\em support} vertex. A {\em claw-free} graph is a graph that does not have a star $K_{1,3}$ as an induced subgraph. Let $A$ and $B$ be two subsets of $V(G)$. We consider $[A,B]$ as the set of edges having one end point in $A$ and the other in $B$. For $S\subseteq V(G)$ we denote by $G[S]$ a subgraph induced by $S$.

A set $S\subseteq V(G)$ is a {\em dominating set} (resp. a {\em total dominating set}) if each vertex in $V(G)\backslash S$ (resp. $V(G)$) has at least one neighbor in $S$. The {\em domination number} $\gamma(G)$ (resp. {\em total domination number} $\gamma_{t}(G)$) is the smallest possible cardinality of a dominating set (resp. total dominating set) in $G$. For more information on domination and total domination we suggest the books \cite{hhs} and \cite{book-total-dom}, respectively. The domination theory is a very large area in graph theory. The amount of parameters on domination is huge and it daily continues growing up. Variations on the domination property or combinations with other parameters are probably the most used features in the new versions of each new domination invariant. In our work, we deal with three versions of domination which are making both actions at the same time. These three parameters have been relatively recently introduced and not many results concerning them are available. In this sense, it is our goal to make several contributions on their knowledge.

A subset $D\subseteq V(G)$ is said to be a {\em double dominating set} of the graph $G$ if $|N[v]\cap D|\geq2$, for all $v\in V(G)$. The {\em double domination number} of $G$, denoted by $\gamma_{\times2}(G)$, is the smallest possible cardinality of a double dominating set in $G$. The concept of double domination was first introduced by Harary and Haynes in \cite{hh}. Our first parameter of interest is as follows. A {\em double outer-independent dominating set} (DOID set) of $G$ is a double dominating set $D$ in $G$ such that $V(G)\setminus D$ is independent. Moreover, the {\em double outer-independent domination number} (DOID number) $\gamma_{d}^{oi}(G)$ of $G$ is the smallest possible cardinality taken over all DOID sets of $G$. This graph parameter was introduced in \cite{k1}.

Our second goal on this work comes as follows. A {\em $2$-outer-independent dominating set} ($2$OID set) in a graph $G$ is a subset $S\subseteq V(G)$ such that every vertex in $V(G)\setminus S$ has at least two neighbors in $S$, and the set $V(G)\setminus S$ is independent. The {\em $2$-outer-independent domination number} ($2$OID number) $\gamma_{2}^{oi}(G)$ of a graph $G$ is the smallest possible cardinality of a $2$OID set in $G$. $2$-outer-independent domination in graphs was first introduced in \cite{jk} by Jafari Rad and Krzywkowski.

The third and last parameter of interest in this work concerns total domination. That is, a total dominating set $S$ in $G$ is said to be a {\em total outer-independent dominating set} (TOID set) if $V(G)\setminus S$ is independent. The {\em total outer-independent domination number} (TOID number) of $G$, denoted by $\gamma_{t}^{oi}(G)$, is the smallest possible cardinality of a TOID set in $G$. This concept was initiated by Krzywkowski in \cite{k2}.

For the sake of convenience, given any parameter $\vartheta$ in a graph $G$, a set of vertices of cardinality $\vartheta(G)$ is called a $\vartheta(G)$-set.
In this paper, we investigate the three aforementioned graph parameters $\gamma_{t}^{oi}(G)$, $\gamma_{d}^{oi}(G)$ and $\gamma_{2}^{oi}(G)$. Clearly, the first two parameters are well-defined if and only if $G$ has no isolated vertices. Since the definitions of them are close to each other, we rather study them simultaneously.

We can easily observe by the definitions that $\gamma_{t}^{oi}(G)\leq \gamma_{d}^{oi}(G)$ and $\gamma_{2}^{oi}(G)\leq \gamma_{d}^{oi}(G)$. Moreover, $\gamma_{t}^{oi}(G)=\gamma_{d}^{oi}(G)$ when $\delta(G)\geq 2$. Note that $\gamma_{t}^{oi}(G)$ and $\gamma_{2}^{oi}(G)$ are not comparable in general. For example, $\gamma_{t}^{oi}(K_{1,n-1})=2$ and $\gamma_{2}^{oi}(K_{1,n-1})=n-1$ for $n\geq4$. On the other hand, $\gamma_{t}^{oi}(C_{n})=\lceil2n/3\rceil$ (see \cite{k2}), while $\gamma_{2}^{oi}(C_{n})=\lceil n/2\rceil$ for $n\geq3$.

The article is organized as follows. Section \ref{sect-NG} is dedicated to find several Nordhaus-Gaddum type results concerning the three parameters $\gamma_{t}^{oi}(G)$, $\gamma_{d}^{oi}(G)$ and $\gamma_{2}^{oi}(G)$. In Section \ref{sect-complex}, complexity results concerning computation of the values of $\gamma_{d}^{oi}(G)$ and $\gamma_{2}^{oi}(G)$ are given. In the last section, several bounds for the three mentioned parameters are presented.


\section{Nordhaus-Gaddum type inequalities}\label{sect-NG}

Nordhaus and Gaddum \cite{ng} in 1956, gave lower and upper bounds on the sum and product of the chromatic numbers of a graph and its complement in terms of the order of the graph. Since then, bounds on $\Psi(G)+\Psi(\overline{G})$ or $\Psi(G)\Psi(\overline{G})$ are called Nordhaus-Gaddum inequalities, where $\Psi$ is any graph parameter. The search of Nordhaus-Gaddum type inequalities has centered the attention of a large number of investigations, and in domination theory, this has probably been even more emphasized. For more information about this subject the reader can consult \cite{ah}.

The Nordhaus-Gaddum type inequalities for $\Psi(G)+\Psi(\overline{G})$ whether $\Psi\in\{\gamma_{2}^{oi},\gamma_{d}^{oi},\gamma_{t}^{oi}\}$ where given in \cite{jk}, \cite{k1} and \cite{k2}, respectively. In this section, we concentrate on the product style of Nordhaus-Gaddum type inequalities of such parameters.

This section is divided into three subsections in correspondence with the three graph parameters considered in this paper. We begin with the Nordhaus-Gaddum type inequalities for the TOID number of graphs.

\subsection{TOID number}\label{SSTOID}

In what follows we may always assume that both $G$ and $\overline{G}$ have no isolated vertices, otherwise the studied parameter is not well defined. This clearly implies that $n\geq4$. Krzywkowski \cite{k2} proved the following Nordhaus-Gaddum type inequalities for the sum of the TOID numbers of a graph and its complement.

\begin{lemma}\emph{(\cite{k2})}\label{L1}
For every graph $G$ of order $n$ we have $n-1\leq \gamma_{t}^{oi}(G)+\gamma_{t}^{oi}(\overline{G})\leq 2n-1$. The equality of the upper bound holds if and only if $G\in\{C_{4},2P_{2}\}$.
\end{lemma}

The next known result is also necessary for our purposes in this section.

\begin{lemma}\emph{(\cite{k2})}\label{L2}
For a connected graph $G$ we have $\gamma_{t}^{oi}(G)=n-1$ if and only if $G\in\{P_{3},C_{4},C_{5}\}$ or $G$ is a complete graph on at least three vertices.
\end{lemma}

We make use of Lemma \ref{L1} and Lemma \ref{L2} as tools in order to give some Nordhaus-Gaddum type inequalities for the product of the TOID numbers of a graph and its complement, and give the characterizations of all graphs for which the equality of the lower bound and the upper bound holds. For the sake of completeness, we next characterize all graphs $G$ attaining the bound $n-1\leq \gamma_{t}^{oi}(G)+\gamma_{t}^{oi}(\overline{G})$. To this aim, we present a different proof for such bound in the proof of the next theorem.

We first define a family $\Lambda$ of graphs $G$ of order $n\geq 5$ as follows. We begin with a path $P_{2}$ on the vertices $a$ and $b$. Add a set $K$ of $n-2$ new vertices. Add some edges with one end point in $\{a,b\}$ and the other one in $K$, provided that we have a partition $\{A,B,C\}$ of $K$ as follows: $A=\{v\in K: \mbox{$v$ is only adjacent to $a$}\}$, $B=\{v\in K: \mbox{$v$ is only adjacent to $b$}\}$ and $C=\{v\in K: \mbox{$v$ is adjacent to both $a$ and $b$}\}$. Notice that $A,B$ and $C$ must be nonempty as they form a partition.

If $n=4$, then $(G,\overline{G})\in \{(P_{4},P_{4}),(C_{4},2P_{2})\}$ and so, $\gamma_{t}^{oi}(G)\cdot\gamma_{t}^{oi}(\overline{G})$ equals either to $4$ or to $12$, respectively. Thus, in what follows we may assume that $n\geq 5$.

\begin{theorem}\label{T3}
If $G$ is a graph $G$ of order $n\geq5$, then
$$2n-6\leq \gamma_{t}^{oi}(G)\cdot\gamma_{t}^{oi}(\overline{G})\leq (n-1)^{2}.$$
Moreover, there is equality in the lower bound if and only if $G\in \Lambda$ or $\overline{G}\in \Lambda$ and, in the upper bound if and only if $G=C_{5}$.
\end{theorem}

\begin{proof}
Suppose first that $\gamma_{t}^{oi}(G)=\gamma_{t}^{oi}(\overline{G})=n$. It is easy to see that, both $G$ and $\overline{G}$ must be the disjoint copies of $P_{2}$, which is impossible. Therefore, $\gamma_{t}^{oi}(G)\cdot\gamma_{t}^{oi}(\overline{G})\leq n(n-1)$. Suppose now that there exists a graph $G$ of order $n$ for which $\gamma_{t}^{oi}(G)\cdot\gamma_{t}^{oi}(\overline{G})=n(n-1)$. Without loss of generality, we may assume that $\gamma_{t}^{oi}(G)=n$ and $\gamma_{t}^{oi}(\overline{G})=n-1$. This leads to that $G=tP_{2}$ for some integer $t\geq3$, and so, $\overline{G}$ is a graph obtained from the complete graph $K_{2t}$ by removing a perfect matching of it. Thus, we observe that $\gamma_{t}^{oi}(\overline{G})=n-2$, and therefore $\gamma_{t}^{oi}(G)\cdot\gamma_{t}^{oi}(\overline{G})=n(n-2)<n(n-1)$, which is a contradiction. In consequence, both $\gamma_{t}^{oi}(G)$ and $\gamma_{t}^{oi}(\overline{G})$ are at most $n-1$ which implies the upper bound.

Suppose now that the equality of the upper bound holds. Hence, $\gamma_{t}^{oi}(G)=\gamma_{t}^{oi}(\overline{G})=n-1$. If one of $G$ and $\overline{G}$, say $\overline{G}$, is disconnected, then $G$ is connected. Since $n\geq 5$ and $G$ is connected, $G=K_{n}$ or $G=C_5$ by Lemma \ref{L2}. In the first case $\overline{G}=\overline{K_{n}}$ contains isolated vertices and $\gamma_t^{oi}(\overline{K_{n}})$ does not exists. The second case yields that $\overline{G}=C_5=G$.  Conversely, we clearly have the equality of the upper bound when $G=\overline{G}=C_{5}$.

To prove the lower bound, we first show that  	
\begin{equation}\label{E3}
n-1\leq \gamma_{t}^{oi}(G)+\gamma_{t}^{oi}(\overline{G}).
\end{equation}
Notice that this follows from Lemma \ref{L1}, however we present this short proof as we later need the statement from (\ref{E4}).
Let $S$ and $\overline{S}$ be a $\gamma_{t}^{oi}(G)$-set and a $\gamma_{t}^{oi}(\overline{G})$-set, respectively. Since $V(G)\setminus S$ and $V(G)\setminus \overline{S}$ are independent in $G$ and $\overline{G}$, respectively, we have
\begin{equation}\label{E4}
n-|S|-|\overline{S}|\leq  n-|S\cup \overline{S}|=|(V(G)\setminus S)\cap(V(G)\setminus \overline{S})|\leq1,
\end{equation}
which implies the inequality (\ref{E3}).

Minimizing $\gamma_{t}^{oi}(G)\cdot\gamma_{t}^{oi}(\overline{G})$ subject to $\gamma_{t}^{oi}(G)+\gamma_{t}^{oi}(\overline{G})=n-1$, we have $(\gamma_{t}^{oi}(G),\gamma_{t}^{oi}(\overline{G}))=(2,n-3)$ or $(\gamma_{t}^{oi}(G),\gamma_{t}^{oi}(\overline{G}))=(n-3,2)$. Therefore,
\begin{equation}\label{E5}
\gamma_{t}^{oi}(G)\cdot\gamma_{t}^{oi}(\overline{G})\geq2n-6,
\end{equation}
as desired.

The equality holds in (\ref{E5}) if and only if $(\gamma_{t}^{oi}(G),\gamma_{t}^{oi}(\overline{G}))=(2,n-3)$ or $(\gamma_{t}^{oi}(G),\gamma_{t}^{oi}(\overline{G}))=(n-3,2)$. Without loss of generality, we may assume that $\gamma_{t}^{oi}(G)=2$ and $\gamma_{t}^{oi}(\overline{G})=n-3$. Let $S=\{a,b\}$ be a $\gamma_{t}^{oi}(G)$-set. Since $\gamma_{t}^{oi}(G)=2$, $G$ is isomorphic either to
\begin{itemize}
  \item[{\rm (i)}] a star $K_{1,n-1}$ in which its central vertex belongs to $S$, or
  \item[{\rm (ii)}] a graph obtained from a path $P_{2}$ on the vertices $a$ and $b$, by adding a set $K$ of $n-2$ new vertices and some edges with one end point in $S$, and the other in $K$, such that each vertex in $K$ has at least one neighbor in $S$ and both $a$ and $b$ are not pendant vertices.
\end{itemize}
Since $\overline{G}$ has no isolated vertices, the case (i) is impossible. Again, since isolated vertices are not allowed, at least one vertex in $K=V(G)\setminus S$ is not adjacent to $a$, and at least one vertex in $K=V(G)\setminus S$ is not adjacent to $b$. Also, as $\overline{G}[V(G)\setminus S]$ is an $(n-2)$-clique in $\overline{G}$, at most one vertex in $V(G)\setminus S$ does not belong to $\overline{S}$, which is a $\gamma_t^{oi}(\overline{G})$-set. If $V(G)\setminus S\subseteq\overline{S}$, then $|\overline{S}|\geq n-2$, and so, $\gamma_{t}^{oi}(G)\cdot\gamma_{t}^{oi}(\overline{G})\geq 2n-4$, which is a contradiction. Therefore, there exists a unique vertex $x$ in $V(G)\setminus S$ which is not in $\overline{S}$. Thus, the other $n-3$ vertices of $V(G)\setminus S$ belong to $\overline{S}$. Since $|\overline{S}|=n-3$, we have $\overline{S}=(V(G)\setminus S)\setminus\{x\}$ and so, $V(G)\setminus \overline{S}=\{a,b,x\}$. Now, the independence of the set $\{a,b,x\}$ in $\overline{G}$ shows that $x\in V(G)\setminus S$ is adjacent to both $a$ and $b$ in $G$. Consequently, we have partitioned $V(G)\setminus S$ into three nonempty subsets $A,B$ and $C$, such that $A=\{v:N_G(v)=\{a\}\}$, $B=\{v:N_G(v)=\{b\}\}$ and $C=\{v:N_G(v)=\{a,b\}\}$. This guarantees that $G\in\Lambda$.

Conversely, assume that $G\in \Lambda$ (consider the same terminology of its definition). Hence, $\{a,b\}$ is a $\gamma_{t}^{oi}(G)$-set and let $x\in C$. Observe that $V(G)\setminus\{a,b,x\}$ is a TOID set in $\overline{G}$ of cardinality $n-3$. Therefore, $\gamma_{t}^{oi}(G)\cdot\gamma_{t}^{oi}(\overline{G})\leq 2n-6$, which completes the proof.
\end{proof}

We now construct a family $\Phi$ of graphs $G$ as follows. Consider a complete graph $K_{p}$ on $3\leq p\leq n-2$ vertices and let $x\in V(K_{p})$. Add a set $J$ with $n-p$ new vertices. Add some edges with one end point in $V(K_{p})\setminus\{x\}$ and the other in $J$, such that
\begin{itemize}
  \item[{\rm(i)}] every vertex in $J$ is adjacent to at least one vertex in $V(K_{p})\setminus\{x\}$, and
  \item[{\rm(ii)}] every vertex in $V(K_{p})\setminus\{x\}$ is adjacent to at most $|J|-1$ vertices in $J$.
\end{itemize}
Let $G$ be the obtained graph. We observe that $V(K_{p})\setminus\{x\}$ is a TOID set in $G$.

\begin{theorem}\label{T4}
For a graph $G$ of order $n$ we have $\gamma_{t}^{oi}(G)+\gamma_{t}^{oi}(\overline{G})=n-1$ if and only if $G\in \Phi$.
\end{theorem}

\begin{proof}
Notice first that for $n\in\{1,2,3\}$ at least one of $\gamma_{t}^{oi}(G)$ and $\gamma_{t}^{oi}(\overline{G})$ does not exists. For $n=4$ both $\gamma_{t}^{oi}(G)$ and $\gamma_{t}^{oi}(\overline{G})$ exists only when $(G,\overline{G})\in \{(P_{4},P_{4}),(C_{4},2P_{2})\}$. Obviously, in these two cases $G$ does not belong to $\Phi$ and $\gamma_{t}^{oi}(G)+\gamma_{t}^{oi}(\overline{G})\neq n-1$. So, we may assume that $n\geq 5$.

Suppose first that $G\in \Phi$. In what follows we prove that $\overline{G}\in \Phi$, as well.\\

\noindent
\textit{Claim 1.} \textit{$G\in \Phi$ if and only if $\overline{G}\in \Phi$.}\\

\noindent
\textit{Proof of Claim 1.} By symmetry it only suffices to show that $G\in \Phi$ implies $\overline{G}\in \Phi$. Note that $\overline{G}[J]$ is a clique in $\overline{G}$. Since $J\cup\{x\}$ is independent in $G$, the vertex $x$ is adjacent to all the other vertices in $J$, when considered as a subset of vertices of the graph $\overline{G}$. So, $\overline{G}[J\cup\{x\}]$ is a clique in $\overline{G}$ of order $n-p+1$. On the other hand, $V(K_{p})$ is an independent set in $\overline{G}$. Thus, $V(K_{p})=V(\overline{G})\setminus J$ is independent in $\overline{G}$. Taking into account the conditions (ii) and (i), respectively, we infer that in $\overline{G}$
\begin{itemize}
  \item[{\rm(i$'$)}] every vertex $v\in V(K_{p})\setminus\{x\}$ has at least one neighbor in $J$; and
  \item[{\rm(ii$'$)}] every vertex in $J$ is adjacent to at most $p-2=|V(K_{p})\setminus\{x\}|-1$ vertices in $V(K_{p})\setminus\{x\}$.
\end{itemize}
Note that $J$, $x$, $V(K_{p})\setminus\{x\}$ in $\overline{G}$ are corresponding to $V(K_{p})\setminus\{x\}$, $x$ and $J$ in $G$, respectively, which shows that $\overline{G}\in \Phi$. ($\square$)

The above argument guarantees that for $G\in \Phi$ also $\overline{G}\in\Phi$. Thus, set $V(K_{p})\setminus\{x\}$ is a TOID set in $G$ and $J$ is a TOID set in $\overline{G}$. Therefore,
$$n-1\leq \gamma_{t}^{oi}(G)+\gamma_{t}^{oi}(\overline{G})\leq|V(K_{p})\setminus\{x\}|+|J|=n-1,$$
implying the equality.

Conversely, suppose that the equality holds. Let $S$ and $\overline{S}$ be a $\gamma_{t}^{oi}(G)$-set and a $\gamma_{t}^{oi}(\overline{G})$-set, respectively. If $(V(G)\setminus S)\cap(V(G)\setminus \overline{S})=\emptyset$, then $\gamma_{t}^{oi}(G)+\gamma_{t}^{oi}(\overline{G})\geq n$ by (\ref{E4}), a contradiction. Again by (\ref{E4}), there exists a unique vertex $w\in(V(G)\setminus S)\cap(V(G)\setminus \overline{S})$. We infer that for $A=(V(G)\setminus S)\setminus\{w\}\subseteq\overline{S}$ and $B=(V(G)\setminus\overline{S})\setminus\{w\}\subseteq S$. We have
$$n-1=2n-|S|-|\overline{S}|-2=|A|+|B|\leq|S|+|\overline{S}|=n-1.$$
Therefore, $A=\overline{S}$ and $B=S$. Moreover, $\overline{G}[A]$ and $G[B]$ are cliques in $\overline{G}$ and $G$, respectively. On the other hand, since $B\cup\{w\}$ is independent in $\overline{G}$, the vertex $w$ is adjacent to all the vertices in $B=S$ in $G$. Since $A=\overline{S}$ is a clique in $\overline{G}$, $A$ is independent in $G$. Since $S$ is a dominating set in $G$, each vertex from $A$ has at least one neighbor in $S$. If there exists a vertex $u\in S$ adjacent to all vertices in $A$, then $u\in B$ has no neighbor in $\overline{S}$ when considered as a dominating set in $\overline{G}$, which is a contradiction. Thus, every vertex in $S$ has at most $|A|-1=p-2$ neighbors in $A$.

We now observe that $w$, $G[B\cup\{w\}]$ and $A$ are corresponding to $x$, $K_{p}$ and $J$, respectively, in the descriptions of the members of $\Phi$. Thus, $G\in \Phi$ and the proof is completed.
\end{proof}

\subsection{2OID number}



In this part we present Nordhaus-Gaddum inequalities for the product of 2OID numbers of $G$ and $\overline{G}$.
In order to characterize all graphs attaining the lower bound in the next theorem, we consider the family $\Psi$ containing all graphs $G$ of order $n\geq6$ obtained as follows. Consider the complete graph $K_{4}$ and let $x\in V(K_{4})$. Add a set $R$ of $n-4$ new vertices. Add some edges with one end point in $V(K_{4})\setminus\{x\}$ and the other in $R$ such that
\begin{itemize}
  \item[{\rm(i)}] every vertex in $R$ is adjacent to at least two vertices in $V(K_{4})\setminus\{x\}$, and
  \item[{\rm(ii)}] every vertex in $V(K_{4})\setminus\{x\}$ is adjacent to at most $|R|-2$ vertices in $R$ (see Figure \ref{fig-psi} for an example).
\end{itemize}

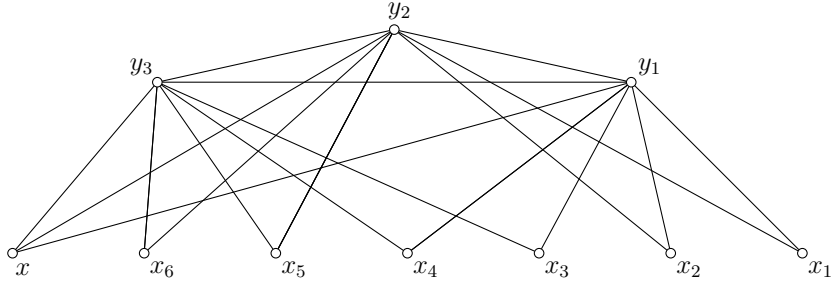
\begin{figure}[ht]
 \centering
\begin{tikzpicture}[scale=.35, transform shape]
\node [draw, shape=circle] (x1) at  (5,7) {};
\node [draw, shape=circle] (x2) at  (0,7) {};
\node [draw, shape=circle] (x3) at  (-5,7) {};

\node [draw, shape=circle] (x4) at  (-10,7) {};
\node [draw, shape=circle] (x5) at  (-15,7) {};
\node [draw, shape=circle] (x6) at  (-20,7) {};
\node [draw, shape=circle] (x7) at  (-25,7) {};

\node [draw, shape=circle] (y1) at  (-1.5,13.5) {};
\node [draw, shape=circle] (y2) at  (-10.5,15.5) {};
\node [draw, shape=circle] (y3) at  (-19.5,13.5) {};

\draw (y1)--(y2)--(y3)--(y1);
\draw (x1)--(y1)--(x2);
\draw (x1)--(y2)--(x5);
\draw (x2)--(y2)--(x5);
\draw (x3)--(y1)--(x4);
\draw (x3)--(y3)--(x6);
\draw (y1)--(x4)--(y3);
\draw (y3)--(x5)--(y2);
\draw (y3)--(x6)--(y2);
\draw (y3)--(x7)--(y2);
\draw (x7)--(y1);

\node [scale=2.4] at (-.8,14.1) {\large $y_{1}$};
\node [scale=2.4] at (-10.3,16.2) {\large $y_{2}$};
\node [scale=2.4] at (-20.1,14.1) {\large $y_{3}$};
\node [scale=2.4] at (5.7,6.4) {\large $x_{1}$};
\node [scale=2.4] at (.7,6.4) {\large $x_{2}$};
\node [scale=2.4] at (-4.3,6.4) {\large $x_{3}$};
\node [scale=2.4] at (-9.3,6.4) {\large $x_{4}$};
\node [scale=2.4] at (-14.3,6.4) {\large $x_{5}$};
\node [scale=2.4] at (-19.3,6.4) {\large $x_{6}$};
\node [scale=2.4] at (-24.6,6.4) {\large $x$};

\end{tikzpicture}\vspace{3mm}
  \caption{A member $G$ of $\Psi$ with $\gamma_{2}^{oi}(G)\cdot\gamma_{2}^{oi}(\overline{G})=18$. Note that $\{y_{1},y_{2},y_{3}\}$ is a $2$OID set in $G$ and $\{x_{1},\cdots,x_{6}\}$ is a $2$OID set in $\overline{G}$.}\label{fig-psi}
\end{figure}

\begin{theorem}\label{TB1}
If $G$ is a graph of order $n\geq4$, then
$$3n-12\leq \gamma_{2}^{oi}(G)\cdot\gamma_{2}^{oi}(\overline{G})\leq n(n-1).$$
Moreover, the equality holds for the lower bound if and only if $G\in \Psi$, and for the upper bound if and only if $G\in\{K_{n},\overline{K_{n}}\}$.
\end{theorem}

\begin{proof}
We first prove the upper bound. Note that $\gamma_{2}^{oi}(G)=n$ if and only if $\Delta(G)\leq 1$. Therefore, $\gamma_{2}^{oi}(G)\cdot\gamma_{2}^{oi}(\overline{G})=n^{2}$ results in $n\leq 2$, which is not possible. Therefore, $\gamma_{2}^{oi}(G)\cdot\gamma_{2}^{oi}(\overline{G})\leq n(n-1)$.

Clearly, the equality holds when $G\in\{K_{n},\overline{K_{n}}\}$. Now let $\gamma_{2}^{oi}(G)\cdot\gamma_{2}^{oi}(\overline{G})=n(n-1)$. Without loss of generality, we may assume that $(\gamma_{2}^{oi}(G),\gamma_{2}^{oi}(\overline{G}))=(n,n-1)$. So we have $\Delta(G)\leq1$. If $\Delta(G)=0$, then $G=\overline{K_n}$ and $\overline{G}=K_n$ and we are done. So, suppose that there exists at least one edge $ab$ in $G$. In addition, there exist vertices $c$ and $d$ in $G$ because $n\geq 4$. Clearly, $acbda$ forms a four-cycle in $\overline{G}$ and $V(G)-\{a,b\}$ is a 2OID set of $\overline{G}$. Hence, $\gamma_{2}^{oi}(\overline{G})\leq n-2$, a contradiction. Therefore, $G$ is a graph without edges and $\overline{G}$ is a complete graph.

We now prove the lower bound. Similar to the proof of Theorem \ref{T3}, we have $n-1\leq \gamma_{2}^{oi}(G)+\gamma_{2}^{oi}(\overline{G})$. Moreover, by minimizing $\gamma_{2}^{oi}(G)\cdot\gamma_{2}^{oi}(\overline{G})$ subject to $\gamma_{2}^{oi}(G)+\gamma_{2}^{oi}(\overline{G})=n-1$, we have $(\gamma_{2}^{oi}(G),\gamma_{2}^{oi}(\overline{G}))=(2,n-3)$ or $(\gamma_{2}^{oi}(G),\gamma_{2}^{oi}(\overline{G}))=(n-3,2)$. In what follows, we prove that the case $(\gamma_{2}^{oi}(G),\gamma_{2}^{oi}(\overline{G}))=(2,n-3)$ is impossible (and symmetrically the case $(\gamma_{2}^{oi}(G),\gamma_{2}^{oi}(\overline{G}))=(n-3,2)$ is also not possible). Suppose that $S$ is a $\gamma_{2}^{oi}(G)$-set and that $S'$ is a $\gamma_{2}^{oi}(\overline{G})$-set of cardinality $2$ and $n-3$, respectively.
Since $S=\{a,b\}$ is a $2$OID set in $G$ and $|S|=2$, vertices $a$ and $b$ are adjacent to all vertices in $V(G)\setminus S$ in the graph $G$. Thus, $a$ and $b$ have no neighbor in $S'\setminus \{a,b\}$ in $\overline{G}$. Therefore, $a$ and $b$ must both be in $S'$ where they can be adjacent or not. Because $|S'|=n-3$, three vertices, say $x,y$ and $z$, from $V(G)\setminus S$ are outside of $S'$. But $\overline{G}[V(G)\setminus S]$ is a complete graph since $V(G)\setminus S$ is independent in $G$. Thus, $x,y$ and $z$ form a triangle in $\overline{G}[V(G)\setminus S']$, a contradiction with $S'$ being an outer-independent set.
This argument shows that the minimum value of $\gamma_{2}^{oi}(G)\cdot\gamma_{2}^{oi}(\overline{G})$ subject to $\gamma_{2}^{oi}(G)+\gamma_{2}^{oi}(\overline{G})=n-1$ is at least $3n-12$ by choosing $3$ or $n-4$ for $\gamma_{2}^{oi}(G)$. Therefore,
\begin{equation}\label{EQB1}
3n-12\leq \gamma_{2}^{oi}(G)\cdot\gamma_{2}^{oi}(\overline{G}).
\end{equation}

The proof of the statement ``$\gamma_{2}^{oi}(G)\cdot\gamma_{2}^{oi}(\overline{G})=3n-12$ if and only if $G\in \Psi$'' is somewhat similar to those presented in the proofs of Theorem \ref{T3} and Theorem \ref{T4}. However, for the sake of completeness, we prove it here. Suppose first that $G\in \Psi$. It is easy to see that $S=V(K_{4})\setminus\{x\}$ is a $2$OID set in $G$ and $\overline{S}=V(G)\setminus V(K_{4})$ is a $2$OID set in $\overline{G}$. Therefore, $\gamma_{2}^{oi}(G)\cdot\gamma_{2}^{oi}(\overline{G})\leq|S|\cdot|\overline{S}|=3n-12$, implying the equality.

Now assume the equality in (\ref{EQB1}). Hence, $(\gamma_{2}^{oi}(G),\gamma_{2}^{oi}(\overline{G}))=(3,n-4)$ or $(\gamma_{2}^{oi}(G),\gamma_{2}^{oi}(\overline{G}))=(3,n-4)$. We first assume that $(\gamma_{2}^{oi}(G),\gamma_{2}^{oi}(\overline{G}))=(3,n-4)$. Let $S$ be a $\gamma_{2}^{oi}(G)$-set of cardinality 3 and let $S'$ be a $\gamma_{2}^{oi}(\overline{G})$-set of cardinality $n-4$. Since $|S'|=n-4$, there exists four independent vertices $a,b,c,d$ in $\overline{G}$. Clearly $a,b,c,d$ form a clique in $G$. As $V(G)\setminus S$ is an independent set of $G$, $S$ must be a subset of $\{a,b,c,d\}$. We may choose notation so that $S=\{a,b,c\}$. Moreover, $d$ is not adjacent to any vertex from $S'$.
On the other hand, since $S'$ is a $2$OID set in $\overline{G}$, every vertex in $S$ has at most $|(V(G)\setminus S)\setminus\{d\}|-2$ neighbors in $(V(G)\setminus S)\setminus\{d\}$.

It is easily seen that $d$, $G[S\cup\{d\}]$ and $(V(G)\setminus S)\setminus\{d\}$ are corresponding to $x$, $K_{4}$ and $R$, respectively, in the descriptions of the members of $\Psi$. Therefore, $G\in \Psi$. By symmetry, we conclude that $\overline{G}\in \Psi$ when $(\gamma_{2}^{oi}(G),\gamma_{2}^{oi}(\overline{G}))=(n-4,3)$. But $\overline{G}\in \Psi$ is equivalent to $G\in \Psi$ (this can be verified similarly to Claim $1$ in the proof of Theorem \ref{T4}). This completes the proof.
\end{proof}

\subsection{DOID number}

Like in the Subsection \ref{SSTOID} we assume that both $G$ and $\overline{G}$ have no isolated vertices, and this implies that $n\geq4$.

Krzywkowski \cite{k1} proved that $\gamma_{d}^{oi}(G)+\gamma_{d}^{oi}(\overline{G})=2n$ if and only if $G=P_{4}$. Therefore, for a graph $G$ different from $P_{4}$ we have
\begin{equation}\label{EQB2}
\gamma_{d}^{oi}(G)+\gamma_{d}^{oi}(\overline{G})\leq 2n-1.
\end{equation}

He also showed that the equality holds in (\ref{EQB2}) if and only if $G$ or $\overline{G}$ can be obtained from a complete graph on at least three vertices by adding a path $P_{2}$ (on two new vertices) and joining one of its vertices to all but one vertex of the complete graph. We remark that $n^{2}$ is a trivial upper bound on $\gamma_{d}^{oi}(G)\cdot\gamma_{d}^{oi}(\overline{G})$ and $n(n-1)$ is an upper bound for it, when $G$ is different from the path $P_{4}.$ It is easy to observe that the above necessary and sufficient conditions are also equivalent to $\gamma_{d}^{oi}(G)\cdot\gamma_{d}^{oi}(\overline{G})=n^2$ and $\gamma_{d}^{oi}(G)\cdot\gamma_{d}^{oi}(\overline{G})=n(n-1)$, respectively.

A result concerning bounding the product of $\gamma_{d}^{oi}(G)$ and $\gamma_{d}^{oi}(\overline{G})$ from below can be obtained  along similar lines to those in the previous subsections. Indeed, we can bound $\gamma_{d}^{oi}(G)+\gamma_{d}^{oi}(\overline{G})$ from below by $n-1$ and derive the lower bound $3n-12$ for $\gamma_{d}^{oi}(G)\cdot\gamma_{d}^{oi}(\overline{G})$. Moreover, similar to Theorem \ref{TB1} we have the equality if and only if $G\in \Psi$.


\section{Computational complexity}\label{sect-complex}

The problem of computing the TOID number of graphs was proved to be NP-hard in \cite{CHSY}. At next we prove that the complexity of the analogous problems for the $2$OID and the DOID numbers are not having a simpler complexity. That is, we next prove the NP-hardness of both problems. To this end, we first analyze the following decision problems.

$$\begin{tabular}{|l|}
  \hline
  \mbox{$2$-OUTER-INDEPENDENT DOMINATION PROBLEM (2OID problem for short)}\\
  \mbox{INSTANCE: A graph $G$ of order $n\ge 3$ and an integer $r,1\le r\le n$.}\\
  \mbox{QUESTION: Is $\gamma_{2}^{oi}(G)\le r$?}\\
  \hline
\end{tabular}$$

$$\begin{tabular}{|l|}
  \hline
  \mbox{DOUBLE OUTER-INDEPENDENT DOMINATION PROBLEM (DOID problem for short)}\\
  \mbox{INSTANCE: A graph $G$ of order $n\ge 3$ and an integer $r,1\le r\le n$.}\\
  \mbox{QUESTION: Is $\gamma_{d}^{oi}(G)\le r$?}\\
  \hline
\end{tabular}$$

In order to study the complexity of the problems above, we make a reduction from the vertex independence problem stated at next, which is known to be NP-complete from \cite{garey}.

$$\begin{tabular}{|l|}
  \hline
  \mbox{INDEPENDENCE PROBLEM}\\
  \mbox{INSTANCE: A graph $G$ of order $n\ge 3$ and an integer $r, 1\le r\le n$.}\\
  \mbox{QUESTION: Is $\alpha(G)\ge r$?}\\
  \hline
\end{tabular}$$

The next results give our desired reductions for the NP-completeness of the 2OID and DOID problems, respectively.

\begin{theorem}
The 2OID problem is NP-complete.
\end{theorem}

\begin{proof}
The problem is clearly in NP since checking that a given set is indeed a $2$OID set can be done in polynomial time. Now, in order to make our reduction, given a graph $G$ of order $n$ with $\delta(G)\geq1$, we construct a graph $G'$ as follows. For any vertex $v_i\in V(G)$ we add a copy of the complete bipartite graph $K_{2,3}$ (we denote such copy as $K_{2,3}^i$ with bipartition sets $A_i,B_i$ such that $|A_i|=2$ and $|B_i|=3$), and join with an edge the vertex $v_i$ with one vertex $v_i'\in A_i$ of $V(K_{2,3}^i)$. We will prove that $\gamma_{2}^{oi}(G')=3n-\alpha(G)$.

First, let $S$ be a $\alpha(G)$-set and let $S'=(V(G)\setminus S)\cup\left(\bigcup_{i=1}^n A_i\right)$. Clearly, every vertex $w_i\in B_i$ for any $i\in \{1,\dots,n\}$ is adjacent to two vertices of $S'$. Also, since $S$ is an independent set and $G$ has no isolated vertices, its complement (in $G$) is a dominating set in $G$. Thus, any vertex $v_i\in S$ has at least two neighbors in $S'$ (at least one in $V(G)\setminus S$ and the vertex $v'_i\in A_i$). Moreover, we observe that $V(G)\setminus S'$ is independent. As a consequence, $S'$ is a $2$OID set in $G'$ which leads to $\gamma_{2}^{oi}(G')\le |S'|=2n+(n-\alpha(G))=3n-\alpha(G)$.

On the other hand, assume that $D$ is a $\gamma_{2}^{oi}(G')$-set. Let $D_G=D\cap V(G)$ and, for every $i\in \{1,\dots,n\}$, let $D_i=D\cap V(K_{2,3}^i)$. If a vertex $u_i\in A_i$ which is not adjacent to a vertex of $G$ does not belong to $D_i$, then every vertex $x\in B_i$ must belong to $D_i$ (since the complement of $D$ is independent), and so, $|D_i|\ge 3$. On the contrary, if such $u_i$ belongs to $D_i$, then at least one other vertex is required in $D_i$ in order to satisfy the $2$-domination condition for the vertices of $K_{2,3}^i$, and so, $|D_i|\ge 2$. In any case, this latter bound will be satisfied. Moreover, we observe that $V(G)\setminus D_G$ is an independent set, which means $\alpha(G)\ge |V(G)\setminus D_G|=n-|D_G|$, or equivalently, $|D_G|\ge n-\alpha(G)$. Finally, we deduce that
$$\gamma_{2}^{oi}(G')=|D|=|D_G|+\sum_{i=1}^n|D_i|\ge n-\alpha(G)+2n=3n-\alpha(G).$$
Therefore, the desired equality follows, and by taking $j=3n-k$, it is readily seen that $\gamma_{2}^{oi}(G')\leq j$ if and only if $\alpha(G)\geq k$, which completes the proof.
\end{proof}

As a consequence of the result above, we conclude that the problem of computing the $2$OID number of graphs is an NP-hard problem. Now, in order to show an analogous result for the DOID number of graph, a similar reduction to the one above for the $2$OID number can be used. The only necessary change in the proof concerns the construction of the graph $G'$. For the DOID problem, we only need to add an edge between the two vertices of the bipartition sets $A_i$ of every $K_{2,3}^i$ ($i\in\{1,\dots,n\}$). In such case, the proof stated above works similarly, and we obtain a similar conclusion. In this sense, we omit the proof of the next result, which also allows to claim that computing the DOID number of graphs is NP-hard.

\begin{theorem}
The DOID problem is NP-complete.
\end{theorem}


\section{Bounds}\label{sect-bound}

Since the problems of computing all three parameters studied in this work are NP-hard, it is desirable to bound their values with respect to several different invariants of the graph. Accordingly, in this section, we bound $\gamma_{2}^{oi}(G),\gamma_{t}^{oi}(G)$ and $\gamma_{d}^{oi}(G)$ from below and above.

\begin{theorem}\label{T13}
Let $G$ be a graph of order $n$ and $\delta(G)\geq 2$. If $G$ is different from a complete graph and an odd cycle, then $\gamma_{2}^{oi}(G)\leq(\Delta(G)-1)n/\Delta(G)$. Moreover, this bound is sharp.
\end{theorem}

\begin{proof}
We have $n-\gamma_{2}^{oi}(G)\leq \alpha(G)$, by the definition. Let $I$ be an $\alpha(G)$-set. Since $\delta(G)\geq2$, every vertex in $I$ has at least two neighbors in $V(G)\setminus I$. Therefore, $V(G)\setminus I$ is a $2$OID set in $G$. This implies that $\gamma_{2}^{oi}(G)\leq n-\alpha(G)$ and hence
\begin{equation}\label{E8}
\gamma_{2}^{oi}(G)=n-\alpha(G).
\end{equation}

It is well-known that $\chi(G)\geq n/\alpha(G)$ (see \cite{w}) where $\chi(G)$ represents the chromatic number of $G$. Therefore, $\alpha(G)\geq n/\Delta(G)$ by Brook's Theorem \cite{b} (which states that for any connected graph $G$ other than a complete graph or an odd cycle, $\chi(G)\leq \Delta(G)$). Now, we derive the upper bound by (\ref{E8}).

The bound is clearly sharp for even cycles. In what follows we present a more general family of graphs attaining the upper bound. We construct a graph $H$ beginning with a cycle $C_{2k}$ on the vertices $v_{1},v_{2},\cdots,v_{2k}$. Next, for every $1\leq i\leq k$, add a complete graph $K_{p}$ ($p\geq 0$) and join the vertices $v_{2i-1}$ and $v_{2i}$ to all vertices of $K_p$. It is easy to see that $\alpha(H)=k$. Moreover, we have $\gamma_{2}^{oi}(H)=k+kp=(p+1)(2k+kp)/(p+2)=(\Delta-1)n/\Delta$.
\end{proof}

As can be readily seen, the difference between any pair of the three parameters studied in this work can be arbitrary large. However, they have a ``better'' behavior while dealing with the family of claw-free graphs, as we next show.

\begin{theorem}\label{T1}
If $G$ is a claw-free graph of order $n$, then the following statements hold.
\begin{itemize}
  \item[{\rm (i)}] $\gamma_{t}^{oi}(G),\gamma_{2}^{oi}(G),\gamma_{d}^{oi}(G)\geq \delta(G) n/(\delta(G)+2)$.
  \item[{\rm (ii)}] If $\delta(G)\geq3$, then $\gamma_{t}^{oi}(G)=\gamma_{2}^{oi}(G)=\gamma_{d}^{oi}(G)$.
\end{itemize}
\end{theorem}

\begin{proof}
(i) Let $S$ be a $\gamma_{2}^{oi}(G)$-set. Since $G$ is claw-free and $V(G)\setminus S$ is independent, every vertex in $S$ has at most two neighbors in $V(G)\setminus S$. Also, every vertex in $V(G)\setminus S$ has $\delta(G)$ neighbors in $S$. Thus,
$$\delta(G)(n-|S|)\leq|[S,V(G)\setminus S]|\leq2|S|.$$
Therefore, $\gamma_{2}^{oi}(G)\geq \delta(G) n/(\delta(G)+2)$.
\noindent Since $\gamma_{d}^{oi}(G)\geq \gamma_{2}^{oi}(G)$ for every graph $G$, this bound is valid also for $\gamma_{d}^{oi}(G)$. The proof for $\gamma_{t}^{oi}(G)$ works among the same lines as for $\gamma_{2}^{oi}(G)$ and we omit the details.

(ii) Clearly, $\gamma_{2}^{oi}(G)\leq\gamma_{d}^{oi}(G)$. Now let $S$ be a $\gamma_{2}^{oi}(G)$-set. Let $v$ be a vertex in $S$. By using the facts that $G$ is claw-free, $V(G)\setminus S$ is independent and $deg(v)\geq3$, we deduce that $v$ has at least one neighbor in $S$. Therefore, $G[S]$ has no isolated vertices. So, $S$ is a DOID set in $G$ as well. We now have $\gamma_{d}^{oi}(G)\leq|S|=\gamma_{2}^{oi}(G)$ implying $\gamma_{2}^{oi}(G)=\gamma_{d}^{oi}(G)$. Moreover, the condition $\delta(G)\geq3$ guarantees that $\gamma_{t}^{oi}(G)=\gamma_{d}^{oi}(G)$ as mentioned in the introduction.
\end{proof}

In the case of cubic claw-free graphs, we have the following result. The lower and upper bounds are immediate results of Theorem \ref{T13} and Theorem \ref{T1}.

\begin{corollary}\label{C2}
For any claw-free cubic graph $G$ of order $n$ different from $K_{4}$,
$$\frac{3n}{5}\leq \gamma_{t}^{oi}(G)=\gamma_{2}^{oi}(G)=\gamma_{d}^{oi}(G)\leq\frac{2n}{3}.$$
\end{corollary}

To see that the upper bound of Corollary \ref{C2} is sharp, we consider the claw-free cubic graph $H_{1}$ depicted in Figure $2$. It is easily seen that $\{x_{1},x_{3},y_{2},y_{3},y_{4},z_{2},z_{3},z_{4}\}$ is a $2$OID set of cardinality $\gamma_{2}^{oi}(H)=2n/3$.
Moreover, to see the tightness of the lower bound of Corollary \ref{C2}, we consider the claw-free cubic graph $H_{2}$ depicted in Figure \ref{fig-reduction}. It is easily seen that $\{u_{1},u_{2},u_{3},u_{4},u_{5},u_{6}\}$ is a $2$OID set of the cardinality $\gamma_{2}^{oi}(H)=3n/5$.

\begin{figure}[ht]
 \centering
\begin{tikzpicture}[scale=.35, transform shape]
\node [draw, shape=circle] (z4) at  (-.3,7.5) {};
\node [draw, shape=circle] (z2) at  (-5.3,7.5) {};
\node [draw, shape=circle] (z1) at  (-2.8,6) {};
\node [draw, shape=circle] (z3) at  (-2.8,9) {};

\node [draw, shape=circle] (y4) at  (-8.7,7.5) {};
\node [draw, shape=circle] (y2) at  (-13.7,7.5) {};
\node [draw, shape=circle] (y1) at  (-11.2,6) {};
\node [draw, shape=circle] (y3) at  (-11.2,9) {};

\node [draw, shape=circle] (x4) at  (-4.5,13.5) {};
\node [draw, shape=circle] (x2) at  (-9.5,13.5) {};
\node [draw, shape=circle] (x1) at  (-7,12) {};
\node [draw, shape=circle] (x3) at  (-7,15) {};

\draw(z1)--(z3);
\draw(z1)--(z2)--(z3)--(z4)--(z1);
\draw(y1)--(y3);
\draw(y1)--(y2)--(y3)--(y4)--(y1);
\draw(x1)--(x3);
\draw(x1)--(x2)--(x3)--(x4)--(x1);
\draw(x4)--(z4);
\draw(z2)--(y4);
\draw(x2)--(y2);

\node [scale=2.4] at (-6.8,11.2) {\large $x_{3}$};
\node [scale=2.4] at (-6.8,15.6) {\large $x_{1}$};
\node [scale=2.4] at (-10.5,13.6) {\large $x_{2}$};
\node [scale=2.4] at (-3.5,13.6) {\large $x_{4}$};

\node [scale=2.4] at (-10.6,9.5) {\large $y_{1}$};
\node [scale=2.4] at (-11.4,5.4) {\large $y_{3}$};
\node [scale=2.4] at (-14.3,7.1) {\large $y_{2}$};
\node [scale=2.4] at (-8.2,6.8) {\large $y_{4}$};

\node [scale=2.4] at (-3.4,9.5) {\large $z_{1}$};
\node [scale=2.4] at (-3,5.4) {\large $z_{3}$};
\node [scale=2.4] at (-5.3,6.8) {\large $z_{2}$};
\node [scale=2.4] at (.3,7) {\large $z_{4}$};


\node [draw, shape=circle] (u1) at  (16,15) {};
\node [draw, shape=circle] (u2) at  (13.1,13.9) {};
\node [draw, shape=circle] (v3) at  (11.5,12.2) {};
\node [draw, shape=circle] (v2) at  (11.6,9.2) {};
\node [draw, shape=circle] (v4) at  (18.7,13.8) {};
\node [draw, shape=circle] (u3) at  (20.3,11.6) {};
\node [draw, shape=circle] (u4) at  (20.2,8.7) {};
\node [draw, shape=circle] (v1) at  (18.4,6.7) {};
\node [draw, shape=circle] (u6) at  (15.7,6.3) {};
\node [draw, shape=circle] (u5) at  (13,7.1) {};

\node [scale=2.4] at (15.9,15.7) {\large $u_{1}$};
\node [scale=2.4] at (12.6,14.5) {\large $u_{2}$};
\node [scale=2.4] at (10.8,12.4) {\large $v_{3}$};
\node [scale=2.4] at (10.7,9.5) {\large $v_{2}$};
\node [scale=2.4] at (12.2,7) {\large $u_{5}$};
\node [scale=2.4] at (15.6,5.5) {\large $u_{6}$};
\node [scale=2.4] at (19,6.1) {\large $v_{1}$};
\node [scale=2.4] at (21,8.5) {\large $u_{4}$};
\node [scale=2.4] at (21.3,11.5) {\large $u_{3}$};
\node [scale=2.4] at (19.4,14.2) {\large $v_{4}$};

\draw(v1)--(u1)--(u2)--(v1)--(u4)--(v4)--(u6)--(v3)--(u5)--(v2)--(u2);
\draw(u1)--(v2);
\draw(v3)--(u3)--(v4);
\draw(u6)--(u5);
\draw(u3)--(u4);

\end{tikzpicture}
  \caption{The claw-free cubic graphs $H_{1}$ and $H_{2}$, respectively.}\label{fig-reduction}
\end{figure}
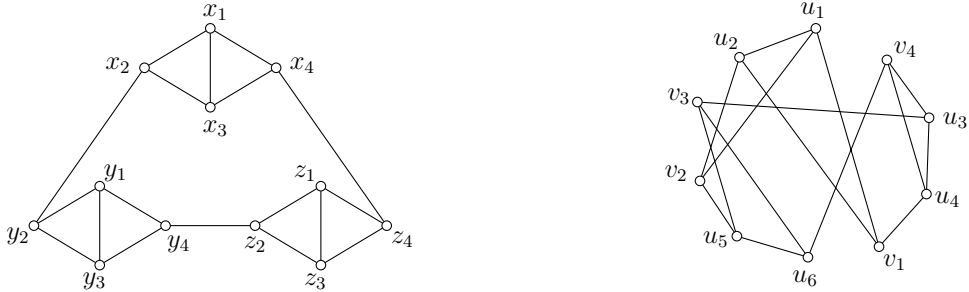

We now turn our attention to triangle-free graphs for which we required the following known result.

\begin{lemma}\emph{(\cite{f})}\label{L3}
If $G$ is a triangle-free graph of order $n$, then $\alpha(G)\geq2n/(3+\Delta(G))$.
\end{lemma}

The following result is an immediate consequence of the equation (\ref{E8}) and Lemma \ref{L3}.

\begin{theorem}\label{T9}
If $G$ is a triangle-free graph of order $n$ with $\delta(G)\geq 2$, then $\gamma_{2}^{oi}(G)\leq\frac{\Delta(G)+1}{\Delta(G)+3}n$.
\end{theorem}

We remark that, Fajtlowicz \cite{f1} studied some classes of triangle-free graphs achieving the equality in the lower bound in Lemma \ref{L3}. This shows that the upper bound in Theorem \ref{T9} is sharp.

Let $L(G)$ and $S(G)$ be the set of leaves and support vertices in the graph $G$, respectively. We define $\delta^{*}(G)$ as the minimum degree taken over all vertices from $V(G)\setminus(L(G)\cup S(G))$, if $L(G)\cup S(G)\subset V(G)$, and 2 otherwise (that is, if $L(G)\cup S(G)=V(G)$). Clearly, $\delta^{*}(G)\geq 2$.

We construct a family $\Omega$ of graphs $G$ as follows. Choose the integers $a,b\geq0$, $r\geq2$ and $p\geq0$ such that $a+b\geq1$ and $2a(r-1)\leq pr$. We begin with the pendant edges $v_{1}v_{1}',\dots,v_{a}v_{a}', v_{a+1}v_{a+1}',\dots,v_{a+b}v_{a+b}'$. Add $p$ isolated vertices $u_{1},\dots,u_{p}$. Join every vertex $u_{i}$ to exactly $r$ vertices in $A=\{v_{1},v_{1}',\dots,v_{a},v_{a}',v_{a+1},\dots,v_{a+b}\}$ such that the degree of all vertices in $B=\{v_{1},v_{1}',\dots,v_{a},v_{a}'\}$ is at least $r$ (this is possible, because $|B|=2a$ and $2a(r-1)\leq pr$). Finally, to obtain the graph $G\in\Omega$, for any $v_{a+i}$ with $1\leq i\leq b$, we add $k_{a+i}\geq 0$ new vertices and join all of them to $v_{a+i}$ by edges. Figure \ref{fig-omega} depicts a fairly representative member of $\Omega$.

\begin{figure}[ht]
 \centering
\begin{tikzpicture}[scale=.35, transform shape]
\node [draw, shape=circle] (v1) at  (0,14) {};
\node [draw, shape=circle] (v1') at  (-3,14) {};
\node [draw, shape=circle] (v2) at  (-7,14) {};
\node [draw, shape=circle] (v2') at  (-10,14) {};
\node [draw, shape=circle] (v3) at  (-14,14) {};
\node [draw, shape=circle] (v3') at  (-17,14) {};
\node [draw, shape=circle] (v4) at  (-21,14) {};
\node [draw, shape=circle] (v4') at  (-24,14) {};

\node [draw, shape=circle] (u1) at  (-3.5,7.5) {};
\node [draw, shape=circle] (u2) at  (-8.5,7.5) {};
\node [draw, shape=circle] (u3) at  (-13.5,7.5) {};
\node [draw, shape=circle] (u4) at  (-18.5,7.5) {};

\node [draw, shape=circle] (w1) at  (-22.5,11) {};
\node [draw, shape=circle] (w2) at  (-24,11) {};
\node [draw, shape=circle] (w3) at  (-25.5,11) {};

\draw (v1)--(v1');
\draw (v2)--(v2');
\draw (v3)--(v3');
\draw (v4)--(v4');
\draw (u1)--(v1)--(u2);
\draw (u1)--(v2)--(u4);
\draw (u1)--(v2')--(u4);
\draw (u2)--(v1')--(u3);
\draw (u2)--(v4);
\draw (v3)--(u3)--(v4);
\draw (u4)--(v1');
\draw (w1)--(v4)--(w2);
\draw (v4)--(w3);

\node [scale=2.4] at (.3,14.6) {\large $v_{1}$};
\node [scale=2.4] at (-2.7,14.7) {\large $v'_{1}$};
\node [scale=2.4] at (-6.7,14.6) {\large $v_{2}$};
\node [scale=2.4] at (-9.7,14.7) {\large $v'_{2}$};
\node [scale=2.4] at (-13.7,14.6) {\large $v_{3}$};
\node [scale=2.4] at (-16.7,14.7) {\large $v'_{3}$};
\node [scale=2.4] at (-20.7,14.6) {\large $v_{4}$};
\node [scale=2.4] at (-23.7,14.7) {\large $v'_{4}$};
\node [scale=2.4] at (-3,6.8) {\large $u_{1}$};
\node [scale=2.4] at (-8,6.8) {\large $u_{2}$};
\node [scale=2.4] at (-13,6.8) {\large $u_{3}$};
\node [scale=2.4] at (-18,6.8) {\large $u_{4}$};
\node [scale=2.4] at (-21.8,10.2) {\large $w_{1}$};
\node [scale=2.4] at (-23.8,10.2) {\large $w_{2}$};
\node [scale=2.4] at (-25.8,10.2) {\large $w_{3}$};

\end{tikzpicture}\vspace{3mm}
  \caption{A member $G$ of $\Omega$ with $(a,b,p,r,k_{3},k_{4})=(2,2,4,3,0,3)$.}\label{fig-omega}
\end{figure}
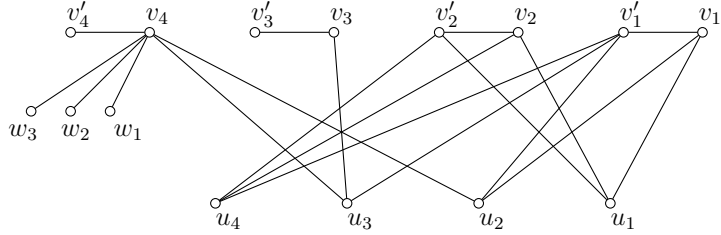

In addition, we recall that a {\em galaxy} is a forest (a collection of trees) in which each component is a star.

\begin{theorem}\label{T2}
If $G$ is a graph of order $n$, size $m$, with $\ell$ leaves and $s$ support vertices, then
$$\gamma_{d}^{oi}(G)\geq\frac{2\delta^{*}(G)n-2m+\ell-s}{2\delta^{*}(G)-1}.$$
Furthermore, the equality holds if and only if $G\in \Omega$.
\end{theorem}

\begin{proof}
If all vertices of $G$ are leaves or support vertices, then $\gamma_{d}^{oi}(G)=n$. Assume $G$ have $c$ components. In such case, the lower bound holds since $n=\ell+s$, $m=n-c$ and $\delta^{*}=2$. Moreover, an easy computation shows that we have the equality if and only if $s=c$, which is equivalent to the situation in which $G$ is a galaxy containing $c$ stars. In this sense, we observe that $G\in \Omega$. Conversely, let $G\in \Omega$. Since all vertices of $G$ are leaves or support vertices, it follows that $p=0$, and so, $G$ is a galaxy containing $c$ stars. Thus, $s=c$ and we have the desired quality.

In what follows, we may assume that there exists a vertex which is neither a leaf nor a support vertex. Let $S(G)=\{w_{1},\dots,w_{s}\}$, and let $\ell_{w_{i}}$ be the number of leaves adjacent to the support vertex $w_{i}$, for each $1\leq i\leq s$. Suppose that $G'$ is the graph obtained from $G$ by removing $\ell_{w_{i}}-1$ leaves from the support vertex $w_{i}$, for each $1\leq i\leq s$. Clearly $n'=n-\ell+s$ and $m'=m-\ell+s$, where $n'$ and $m'$ are the order and the size of $G'$, respectively.

Since all leaves and support vertices belong to each DOID set in a graph, it follows that
\begin{equation}\label{E1}
\gamma_{d}^{oi}(G')+\ell-s=\gamma_{d}^{oi}(G).
\end{equation}
Moreover,
\begin{equation}\label{E2}
\delta^{*}(G')=\delta^{*}(G).
\end{equation}

Now let $D'$ be a $\gamma_{d}^{oi}(G')$-set. Since $E(G[V(G')\setminus D'])=\emptyset$, we have $m'=|E(G[D'])|+|[D',V(G')\setminus D']|$. Also, since $L(G')\cup S(G')\subseteq D'$, every vertex in $V(G')\setminus D'$ has at least $\delta^{*}(G')$ neighbors in $D'$. Therefore, $|[D',V(G')\setminus D']|\geq \delta^{*}(G')(n'-|D'|)$. Moreover, $|E(G[D'])|\geq|D'|/2$ because $G'[D']$ has no isolated vertices. Therefore, $2m'\geq2\delta^{*}(G')(n'-|D'|)+|D'|$. This implies that
$$\gamma_{d}^{oi}(G')=|D'|\geq\frac{2n'\delta^{*}(G')-2m'}{2\delta^{*}(G')-1}.$$
By (\ref{E1}) and (\ref{E2}), we have
$$\gamma_{d}^{oi}(G)\geq\frac{2n'\delta^{*}(G)-2m'}{2\delta^{*}(G)-1}+\ell-s=\frac{2n\delta^{*}(G)-2m+\ell-s}{2\delta^{*}(G)-1},$$
as desired.

Suppose that $G\in \Omega$. Let $L_{a+i}$ be the set of leaves adjacent to $v_{a+i}$ different from $v_{a+i}'$. Note that
\begin{equation}\label{E7}
\begin{array}{lcl}
n=p+2a+2b+\sum_{i=1}^{b}|L_{a+i}|,\ m=pr+a+b+\sum_{i=1}^{b}|L_{a+i}|, \delta^{*}(G)=r\\
\ \ \ \ \ \ \ \ \ \ \ \ \ \ \ \ \ \ \ \ \ \ \ \ \ \ \ \ \ \ \ \mbox{and} \sum_{i=1}^{b}|L_{a+i}|=\ell-s.
\end{array}
\end{equation}

On the other hand, it follows from the structure of $G$ that
$$D=\{v_{1},v_{1}',\cdots,v_{a},v_{a}',v_{a+1},v_{a+1}',\cdots,v_{a+b},v_{a+b}'\}\cup \left(\bigcup_{i=1}^{b}L_{a+i}\right)$$
is a DOID set in $G$. Therefore, $$\gamma_{d}^{oi}(G)\leq|D|=n-p=\frac{2n\delta^{*}(G)-2m+\ell-s}{2\delta^{*}(G)-1},$$
by using the values of $n$, $m$, $\delta^{*}(G)$, $\ell$ and $s$ from (\ref{E7}). This implies the desired equality.

Conversely, assume that we have the equality in the bound. Since $L(G)\cup S(G)\subseteq S$, it suffices to show that $G'\in \Omega$. Since $\gamma_{d}^{oi}(G)=\gamma_{d}^{oi}(G')+\ell-s$, we have $$\gamma_{d}^{oi}(G')=\frac{2n'\delta^{*}(G')-2m'}{2\delta^{*}(G')-1}.$$
Therefore,
$$|[D',V(G')\setminus D']|=\delta^{*}(G')(n'-|D'|)\ \mbox{and}\ |E(G[D'])|=|D'|/2$$
necessarily. This shows that $E(G[D'])$ is a disjoint union of copies of $P_{2}$, and every vertex in $V(G')\setminus D'$ has exactly $\delta^{*}(G')=\delta^{*}(G)=r$ neighbors in $D'\setminus L(G')$. Moreover, $G'[V(G')\setminus D']$ is edgeless, by the definition. Let us denote $E(G[D'])$ by $\{v_{1}v_{1}',\cdots,v_{a}v_{a}',v_{a+1}v_{a+1}',\cdots,v_{a+b}v_{a+b}'\}$ in which $L(G')=\{v_{a+1}',\cdots,v_{a+b}'\}$ and the set of vertices of $G'[V(G')\setminus D']$ by $\{u_{1},\cdots,u_{p}\}$. Since $D'\neq \emptyset$, we have $a+b\geq1$. Moreover, $deg(v_{i}),deg(v_{i}')\geq r$ for all $1\leq i\leq a$. This implies that $pr=|[D',V(G')\setminus D']|\geq2a(r-1)$. Therefore, $G'\in \Omega$, which completes the proof.
\end{proof}

Note that the lower bound $f(\delta^{*}(G))$ given in Theorem \ref{T2} is an increasing function on $\delta^{*}(G)$. Since $\delta^{*}(G)\geq2$, we have a lower bound $f(2)$ on $\gamma_{d}^{oi}(G)$. By choosing $n-1$ instead of $m$ in $f(2)$, for any tree $T$ of order $n$ and size $m$, we immediately conclude the following result.

\begin{corollary}\emph{(\cite{k})}\label{C1}
If $T$ is a tree of order $n$ with $\ell$ leaves and $s$ support vertices, then $\gamma_{d}^{oi}(T)\geq(2n+\ell-s+2)/3$.
\end{corollary}

Jafari Rad and Krzywkowski \cite{jk} proved the following lower bound on the $2$OID number of a graph.

\begin{proposition}\emph{(\cite{jk})}\label{P1}
For any graph $G$ of order $n$ and size $m$, $\gamma_{2}^{oi}(G)\geq n-m/2$. 
\end{proposition}

Also, they posed the following open problem.

\begin{p}\emph{(\cite{jk})}
Characterize all graphs $G$ for which, $\gamma_{2}^{oi}(G)=n-m/2$.
\end{p}

In order to solve this problem, we introduce the family $\mathcal{G}$ of graphs $G$ as follows. Let $P$ and $Q$ be two sets of vertices with $|P|\geq 1$ and $|Q|\geq 2$. We define the graph $G$ by $V(G)=P\cup Q$ and joining every vertex in $P$ to exactly two vertices in $Q$.

\begin{proposition}
For any graph $G$ of order $n$ and size $m$, $\gamma_{2}^{oi}(G)=n-m/2$ if and only if $G\in \mathcal{G}$.
\end{proposition}

\begin{proof}
We first present a short proof for the already known bound $\gamma_{2}^{oi}(G)\geq n-m/2$. Let $S$ be a $\gamma_{2}^{oi}(G)$-set. Since $V(G)\setminus S$ is independent and every vertex in $V(G)\setminus S$ has at least two neighbors in $S$, we have
\begin{equation}\label{E6}
m=|E(G[S])|+|[S,V(G)\setminus S]|\geq2(n-|S|).
\end{equation}
The lower bound now immediately follows from (\ref{E6}). Let $G\in \mathcal{G}$. Clearly, $Q$ is a $2$OID set in $G$. Hence, $\gamma_{d}^{oi}(G)\leq|Q|=n-|P|=n-m/2$ and this implies the equality.

Conversely, suppose that $\gamma_{2}^{oi}(G)=n-m/2$, which yields the equality of (\ref{E6}) as well. By the definition, $V(G)\setminus S$ is independent and every vertex from $V(G)\setminus S$ has at least two neighbors in $S$. Thus $|[S,V(G)\setminus S]|\geq2(n-|S|)$. Because the equality holds in (\ref{E6}), graph $G[S]$ is without edges and every vertex from $V(G)\setminus S$ has exactly two neighbors in $S$.
It is now easy to see that, $S$ and $V(G)\setminus S$ can be taken as $Q$ and $P$, respectively, in the descriptions of all graphs in $\mathcal{G}$. Thus, $G\in \mathcal{G}$.
\end{proof}

The $2$OID number of any bipartite graph $G$ with minimum degree at least two, can be bounded from above by half of the order. In the next theorem, we bound $\gamma_{t}^{oi}(G)$ and $\gamma_{d}^{oi}(G)$ for bipartite graphs $G$ from above.

\begin{theorem}\label{T6}
If $G$ is a bipartite graph of order $n$ with $\delta(G)\geq 2$, then
$$\gamma_{t}^{oi}(G)=\gamma_{d}^{oi}(G)\leq\frac{n+\gamma_{t}(G)}{2}.$$
Furthermore, this bound is sharp.
\end{theorem}

\begin{proof}
Let $D$ be a $\gamma_{t}(G)$-set. Suppose that $I$ is a maximum independent set in $G[V(G)\setminus D]$. We define $A=\{v\in I: |N(v)\cap D|=1\}$ and $B=\{v\in I: |N(v)\cap D|\geq2\}$. Since $D$ is a dominating set, we have $I=A\cup B$. Since $\delta(G)\geq2$, every vertex in $A$ has at least one neighbor in $(V(G)\setminus D)\setminus I$. Therefore, every vertex in the independent set $I$ has at least two neighbors in $S=D\cup((V(G)\setminus D)\setminus I)$. Since $D$ is a total dominating set in $G$, the subgraph of $G$ induced by $S$ has no isolated vertices. Therefore, $S$ is a $2$OID set in $G$.

On the other hand, $|I|\geq(n-|D|)/2$ because $G[V(G)\setminus D]$ is a bipartite graph. Now we have, $\gamma_{2}^{oi}(G)\leq|S|=n-|I|\leq(n+|D|)/2=(n+\gamma(G))/2$.

That this bound is sharp may be seen as follows. Consider the grid graph (Cartesian product of two paths) $P_{3k}\square P_{2}$ with set of vertices $\{v_{ij}: 1\leq i\leq3k,1\leq j\leq2\}$ in a matrix form, for some positive integer $k$. Clearly, $V(P_{3k}\square P_{2})$ can be partitioned into the subsets $V_{i}=\cup_{j=1}^{2}\{v_{(3i-2)j},v_{(3i-1)j},v_{(3i)j}\}$, for $1\leq i\leq k$. At least four vertices in each $V_{i}$ must belong to every DOID set in $P_{3k}\square P_{2}$. Therefore, $\gamma_{2}^{oi}(P_{3k}\square P_{2})\geq4k$. On the other hand, $\cup_{i=1}^{k}\{v_{(3i-2)2},v_{(3i-1)1},v_{(3i-1)2},v_{(3i)1}\}$ is a DOID set in $P_{3k}\square P_{2}$ of cardinality $4k$ and so, $\gamma_{2}^{oi}(P_{3k}\square P_{2})=4k$. Moreover, $A=\cup_{i=1}^{k}\{v_{(3i-1)1},v_{(3i-1)2}\}$ is a $\gamma_{t}(P_{3k}\square P_{2})$-set. Thus $\gamma_{t}(P_{3k}\square P_{2})=|A|=2k$. Therefore, $\gamma_{2}^{oi}(P_{3k}\square P_{2})=(n+\gamma_{t}(P_{3k}\square P_{2}))/2$.
\end{proof}

Note that both the conditions ``being bipartite'' and ``$\delta\geq 2$'' are necessary in Theorem \ref{T6}. For instance, the bound from the theorem do not hold for $K_{1,n-1}$ and $K_{n}$ when $n\geq5$.

\end{document}